\documentclass{amsart}

\usepackage{amsmath}
\usepackage[dvips]{graphicx}
\usepackage{amsfonts}
\usepackage{amssymb}
\usepackage{amsthm}
\usepackage{mathrsfs}
\usepackage[initials]{amsrefs}
\usepackage[draft]{optional}
\usepackage{graphicx, color}
\usepackage{amssymb, paralist}
\usepackage{amsfonts, amsmath, wasysym}
\usepackage{pdfsync}
\usepackage{latexsym}
\usepackage{graphicx, color}
\usepackage{indentfirst}
\usepackage{bm}
\usepackage[active]{srcltx}

\newcommand{\R}{{\mathbb{R}}}

\newtheorem{thm}{Theorem}

\newtheorem{lemma}[thm]{Lemma}
\theoremstyle{remark}
\newtheorem{remark}{Remark}
\theoremstyle{definition}
\newtheorem*{defn}{Definition}
\theoremstyle{definition}
\newtheorem*{ex}{Model Case}
\theoremstyle{definition}

\newtheorem*{claim}{Claim}

\begin{document}

%\title[three rigid inequalities for conformally flat manifolds]{Three rigid inequalities involving ADM mass, capacity 
%and mean curvature, for conformally flat manifolds}
\title[mass-capacity inequalities for conformally flat manifolds]{MASS-CAPACITY INEQUALITIES FOR CONFORMALLY FLAT MANIFOLDS WITH BOUNDARY}

\author{Alexandre Freire}
\author{Fernando Schwartz}

\begin{abstract}    In this paper we prove a mass-capacity inequality and a 
volumetric Penrose inequality for
 conformally flat manifolds, in arbitrary dimensions.  
As a by-product of the proofs, P\'olya-Szeg\"o and Aleksandrov-Fenchel inequalities for mean-convex Euclidean domains are
obtained. For each inequality, the case of equality is characterized.

%The mass-capacity inequality was used by Bray in the proof of the Riemannian Penrose inequality 
%\cites{bray01}. Bray's proof of the mass-capacity inequality uses the positive mass theorem. (In fact, this is 
%apparently the only place in \cite{bray01} where the positive mass theorem is needed.) Our
%proof for conformally flat manifolds  uses only classical arguments, and it
% holds in arbitrary dimensions.  Our result may 
 % be regarded as evidence that the Riemannian Penrose inequality for   conformally flat manifolds
 % can be obtained from arguments of classical linear elliptic theory.
\end{abstract}

\address{Department of Mathematics, University of Tennessee, Knoxville, USA}
\email{freire@math.utk.edu}
\email{fernando@math.utk.edu}
\maketitle

%%%%%%%%%%%%%%%%%%%%%%%%%%%%%%%%%%%%%%%%%%%%%
%%%%%%%%%%%%%%%%%%%%%%%%%%%%%%%%%%%%%%%%%%%%%

\section{Introduction and Main Results}

%%%%%%%%%%%%%%%%%%%%%%%%%%%%%%%%%%%%%%%%%%%%%
%%%%%%%%%%%%%%%%%%%%%%%%%%%%%%%%%%%%%%%%%%%%%

Inequalities between quasi-local quantities and global quantities have recently generated a fair amount of interest.
Among those, the spacetime Penrose inequality stands out as one of the challenging open problems in 
mathematical relativity. 

 The Riemannian version of the Penrose inequality for three-dimensional manifolds 
was proved by Huisken and Ilmanen \cite{huiskenilmanen}  using
inverse mean curvature flow (for the case of connected horizons),
and by Bray \cite{bray01}  using a conformal flow of the metric (for the general case).
The  argument of  Bray uses the 
 mass-capacity inequality in order
to prove  the monotonicity of the ADM mass along the conformal flow. 

The proof of the mass-capacity inequality
in Bray's work  relies on the positive mass theorem and a modification of the 
 reflection argument of \cite{masood}.  
A related reflection argument (implicitly involving a mass-capacity inequality) was later used 
by Bray and Lee \cite{braylee} (also using the positive mass theorem) in order to prove
the Riemannian Penrose inequality for dimensions less than eight. For
the case of a connected boundary, but now only for dimension $3$,  Bray and Miao 
 \cite{braymiao} gave a proof of the mass-capacity inequality which uses the 
 monotonicity of the Hawking mass
 along the inverse mean curvature flow \cite{huiskenilmanen} instead of the positive 
 mass theorem.  
 
 Our proof of the mass-capacity inequality (as well as the proof of the other inequalities) 
 uses only classical arguments and works
 in arbitrary dimensions.
 
 \begin{defn}
A {\it conformally flat} manifold with boundary, or CF-manifold for short,  is a manifold
$(M^{n},g), n\ge 3$, isometric to the complement of a smooth bounded open set (not necessarily connected) 
$\Omega\subset\mathbb{R}^n$ together with a 
conformally flat metric $g_{ij}=u^{\frac 4{n-2}}\delta_{ij}$, where $u>0$ is smooth, and so that:
\begin{itemize}
\item $g$ is asymptotically flat, with non-negative scalar curvature (i.e. $\Delta_{0}u\le 0$), and 
normalized so that $u\to 1$ at infinity,
\item $\Sigma=\partial \Omega$ is mean-convex  with respect to the Euclidean
metric (i.e. $H_{0}>0)$,
\item $\Sigma=\partial M$ is
 minimal with respect
 to the metric $g$ (i.e. $H_{g}=0$).
\end{itemize}
 \end{defn}

The main results of this paper are a mass-capacity inequality and a
  volumetric Penrose inequality for CF-manifolds  
 (the latter is an improved version of the inequality of \cite{schwartz}), as well as
  a P\'olya-Szeg\"o and an Aleksandrov-Fenchel inequality for
Euclidean domains.  The precise statements are the following. (See Section 2 for  definitions.)

\begin{thm} \label{main}\label{sch}
 Let $(M,g)$ be a CF-manifold as above, and  let $m$ denote its ADM mass. Let $\alpha=\min_{\Sigma}u$.
Assume either
\begin{itemize}
 \item[(i)] $\alpha\geq 2$ or 
\item[(ii)] $\alpha<2$, the oscillation of $u$ on $\Sigma$ is small and either $n<8$ or $\Sigma$ is outer-minimizing in $\Omega^c$.
\end{itemize} 

\noindent Then we have:
 \begin{itemize}
 \item[(a)] Mass-capacity inequality: $$m\ge C_g(\Sigma),$$
where $C_g(\Sigma)$ denotes the capacity of $\Sigma$ in $(M,g)$. 
\item[(b)] Volumetric Penrose inequality: $$m\geq 2\left(\frac{V_0}{\beta_n}\right)^{\frac{n-2}n},$$
where $V_{0}$ is the Euclidean volume of $\Omega$, and $\beta_{n}$ is the volume of the 
Euclidean unit $n$-ball.  
\item[(c)] Rigidity: in case (i), equality holds in (a) or (b) if and only if $g$ is the Riemannian Schwarzschild metric. In case (ii), both inequalities are strict.
\end{itemize}
\end{thm}

\begin{remark} In (ii) above, ``small" depends on $\alpha$, and is made precise in the statement of Theorem 5(IV).
\end{remark}

\begin{thm}\label{afr}
Let $\Omega\subset\R^{n}$ be a smooth bounded domain (not necessarily connected) with mean-convex boundary 
$\Sigma=\partial\Omega$.  Assume $\Sigma$ is outer-minimizing in $\Omega^c$. 
Denote by $V_{0}$ its volume, and by $ A_{0},~ H_{0}>0$ the area and  mean curvature of $\Sigma$, respectively. 
 Then we have:
\begin{itemize}
\item[(a)] P\'olya-Szeg\"o inequality:
$$C_{0}(\Sigma)\le  \frac 1{(n-1)\omega_{n-1}}\int_{\Sigma}H_0d\sigma_0, $$
with equality achieved if and only if $\Omega$ is a  round ball.
\item[(b)] Aleksandrov-Fenchel inequality: 
$$\frac 1{(n-1)\omega_{n-1}}\int_{\Sigma}H_0d\sigma_0\ge \left(\frac{A_{0}}{\omega_{n-1}}\right)^{\frac{n-2}{n-1}},$$
with equality achieved if and only if $\Omega$ is a  round ball.
\end{itemize}
\end{thm}

The proof of the above results follows from Theorem \ref{two}  below, which relies on classical arguments 
including
Huisken and Ilmanen's inverse mean curvature flow for arbitrary dimensions \cites{huiskenilmanen,hi08}.
It is important to note that $\Sigma$ is not assumed to be connected in either of the above results.\\

The novelty in part (a) of Theorem \ref{main} is that it does not use the positive mass theorem and applies in all
dimensions.
Compared to \cite{schwartz}, the novelty in part 
(b) of  Theorem \ref{sch} is that it is a sharp estimate and includes a rigidity statement. 
In the case of convex domains, part (a) of Theorem \ref{afr} is related to a classic result of 
P\'olya-Szeg\"o \cite{szego}.  Our result for mean-convex domains is more general.  Part
(b) of Theorem \ref{afr} is included in the family of 
classical Aleksandrov-Fenchel inequalities for the cross-sectional volumes of convex domains. 
These were generalized to the case of star-shaped $k$-convex domains in \cite{guanli}. 
Our result for outer-minimizing, mean-convex domains (or 1-convex, $k=1$) does not require the domain 
to be star-shaped (or even connected), hence it is more general.  \\

{\bf Acknowledgments.} We are indebted to Jeff Jauregui for pointing out a mistake in the rigidity statement in part (III) of theorem 5, in an earlier version of the paper.  
We would like to thank the anonymous referee, whose remarks led to improvements in the rigidity
proofs.

%%%%%%%%%%%%%%%%%%%%%%%%%%%%%%%%%%%%%%%%%%%%%
%%%%%%%%%%%%%%%%%%%%%%%%%%%%%%%%%%%%%%%%%%%%%

\section{Preliminaries}

%%%%%%%%%%%%%%%%%%%%%%%%%%%%%%%%%%%%%%%%%%%%%
%%%%%%%%%%%%%%%%%%%%%%%%%%%%%%%%%%%%%%%%%%%%%

Let $(M^n,g), n\ge 3$ be a complete, non-compact Riemannian manifold with boundary $\Sigma=\partial M$. Here, we don't assume $\Sigma$ is connected. For simplicity, let us assume $M$ has only one end, $\mathcal{E}$. 
Such a manifold is said to be {\it asymptotically flat} if, outside a compact set, 
 $(M,g)$ is diffeomorphic to the complement of a ball in Euclidean space, and in the coordinates
 given by this diffeomorphism the metric has the asymptotic decay
$$|g-\delta|=O(|x|^{-p}),\quad |\partial g|=O(|x|^{-p-1}),\quad |\partial^2g|=O(|x|^{-p-2}),$$
where $p>\frac{n-2}2.$  Furthermore, we require $(M,g)$ to have
  integrable scalar curvature $\int_M|R_g|dV_g<\infty$. 
  
   For these manifolds the ADM mass does not
  depend on the choice of asymptotically flat coordinates and is
  defined by
  \begin{equation}\label{admmass}
  m=m_{ADM}(g)=\frac 1{2(n-1)\omega_{n-1}}\lim_{r\rightarrow \infty} \int_{S_r}\sum_{i,j}(\partial_jg_{ij}-\partial_ig_{jj})\nu^jd\sigma^0_r.
  \end{equation}
Here, $S_r$ is a Euclidean coordinate sphere, $d\sigma^0_r$ is Euclidean surface area.\\

There are several results in the literature which give lower bounds for the ADM mass in terms of geometric
quantities.  For example,
the celebrated {\it positive mass theorem} \cites{schoenyau, witten} (valid for asymptotically flat manifolds {\it without} boundary) asserts
 that if the scalar curvature of $g$ is non-negative and either $3\leq n\leq 7$ or $M$ is spin, then 
 $$m\geq 0,$$ and $m=0$ if and only if the manifold is Euclidean space. \\

Another well-known inequality is the {\it Riemannian Penrose inequality}, which
can be thought of as a refinement of the positive mass theorem.  It asserts that if $M$ has non-negative
scalar curvature and contains a compact
 outermost minimal hypersurface $\Sigma$, then
$$m\geq \frac 12\left(\frac{|\Sigma|}{\omega_{n-1}}\right)^{\frac{n-2}{n-1}},$$
where $|\Sigma|$ denotes the $g$-area of $\Sigma$ and $\omega_{n-1}$ is the
volume of the $(n-1)$-dimensional sphere.  Rigidity also holds for the Riemannian
Penrose inequality.  More precisely,
 equality holds above if and only if the manifold is a {\it Riemannian Schwarzschild manifold} of mass $m>0$
$$g_{ij}=\left(1+\frac m{2r^{n-2}}\right)^{\frac 4{n-2}}\delta_{ij},$$
where $g$ is defined outside the ball of radius $R_s:=(\frac m2)^{\frac 1{n-2}}.$
This inequality was proved in \cite{huiskenilmanen} for $n=3$ and connected $\Sigma$ (using inverse mean-curvature flow and monotonicity of the Hawking mass), and in \cite{bray01} for $n=3$ without the connectedness assumption.
The approach of \cite{bray01} was generalized for $3\leq n\leq 7$ in \cite{braylee}, although the rigidity 
statement requires the extra hypothesis that the manifold
be spin. \\

It is natural to wonder if there is a proof of the Riemannian Penrose inequality
in the general conformally flat case that uses only properties of 
superharmonic functions in $\mathbb{R}^n$ (see \cite{brayiga}).   Our work provides  evidence 
in this direction.\\

In what follows we will be using the notion of capacity of hypersurfaces.
The precise definition  is the following.

\begin{defn}
Let $(M,g)$ be a complete, non-compact Riemannian manifold with compact boundary $\Sigma$ and one end ${\mathcal{E}}$.
The {\it capacity} of a hypersurface $\Sigma\subset (M,g)$ is  
$$C_g(\Sigma)=\inf_{\varphi\in M_0^{1}} \left\{\frac 1{(n-2)\omega_{n-1}}\int_M|\nabla_g\varphi|_g^2dV_g\right\},$$
where $M_{0}^{1}$ denotes the set of all smooth functions on $M$ which are exactly
$0$ on $\Sigma$ and  approach
$ 1$ towards infinity in the end ${\mathcal{E}}.$  
We denote by $C_0(\Sigma)$  the Euclidean capacity of a hypersurface   
$\Sigma=\partial \Omega\subset \mathbb{R}^{n}.$ 
\end{defn}

\begin{remark}  The normalization constant of the above definition is chosen so that $C_0(S_R)=R^{n-2}$, where
  $S_R=\partial \mathbb{B}_R$ in $\mathbb{R}^n$.
 \end{remark}
 
\begin{remark} The infimum of the definition is attained by the unique $g$-harmonic function in $M_{0}^{1}$. If the
ambient manifold is Euclidean space, it follows that the harmonic function which realizes the infimum
has the asymptotic expansion
$$\varphi(x)=1-\frac{C_0(\Sigma)}{|x|^{n-2}}+O(|x|^{1-n})\mbox{ as }x\rightarrow \infty.$$
 \end{remark}

\begin{remark}  Changing the boundary conditions we could also define (for $a\neq b$):
$$C^{(a,b)}_g(\Sigma)=\inf_{\varphi\in M_{a}^{b}} \left\{\frac 1{(n-2)\omega_{n-1}}\int_M|\nabla_g\psi|_g^2dV_g\right\},$$
where $M_{a}^{b}$ is defined as above.  Since the map $\psi\mapsto \frac{a-\psi}{a-b}$ defines a bijection
 $M_{a}^{b}\to M_{0}^{1}$ which scales the integral of the square of the gradient by a constant, it follows that
 $C_g^{(a,b)}(\Sigma)=(a-b)^2C_g(\Sigma)$.\\
\end{remark}

In this paper we are interested in the case when $(M,g)$ is a CF-manifold.  Recall from its definition 
 that this means that $M$ is
 diffeomorphic to $\Omega^c:=\mathbb{R}^n\setminus \Omega$, where $\Omega \subset \mathbb{R}^n$  is a smoothly bounded domain (not necessarily connected), and  $g$ is conformal to the Euclidean metric.  That is,
  $g_{ij}=u^{\frac 4{n-2}}\delta_{ij}$ with  $u>0$, and $u\to1$ at Euclidean infinity.\\

     In what follows we denote $\Sigma=\partial \Omega$.  For reasons that will 
     become clear below,  it is 
     convenient to prove our
     results for CF-manifolds which are also {\it harmonically flat at infinity.}
   In our case, this means that the superharmonic function $u$ also satisfies
 $\Delta_0 u=0$  outside a sufficiently large Euclidean ball.  Using an expansion in 
 spherical harmonics we get that
$$u=1+\frac m{2r^{n-2}}+O(r^{1-n}),\quad u_r=-\frac{(n-2)m}2 r^{1-n}+O(r^{-n}),\quad m=m_{ADM}(g).$$ 

  It is well known that 
     changing $m$ by an arbitrarily small amount (and $g$ by a point-wise ratio arbitrarily close to $1$),
      one may assume $g$ is harmonically flat at infinity.  (See e.g. \cite{schoenyau}.)
      For our purposes it is useful to  construct an explicit approximation
      by such metrics.   This will play 
  a role in proving some of the rigidity statements of the main theorem, as we see 
  below.  The construction we present is inspired by a construction of 
  Miao \cite{miao04}.

\begin{lemma}[Approximation by metrics that are harmonically flat at infinity]\label{approxi}
Let $(M,g)$ be a CF-manifold of mass $m$ as above, i.e. $M$ is isometric to
$\Omega^c\subset \R^n$ with
$g=u^{4/(n-2)}\delta_{ij}$.  Then,
there exists a sequence of smooth superharmonic
functions $\{u_k\}$ defined on $\Omega^c\subset\R^n$ so that for $k$ large enough
\begin{enumerate}
\item[(i)] $u_k$ is harmonic outside $\mathbb{B}_{k+\frac1{k}}$, $u_k\to1$ at infinity,
\item[(ii)] $u_k\to u$ uniformly
on compact
subsets of $M$,
\item[(iii)] $\lim_{k\to\infty}m_k=m$, where $m_k$ is the ADM-mass of $u_k^{4/(n-2)}\delta_{ij}.$
%\item[(iv)] $\limsup_{k\to\infty}C_{g_k}(\Sigma)=C_g(\Sigma)$, where 
%$C_{g_k}(\Sigma)$ is the capacity of the boundary of $M$ with respect to the metric
 %$u_k^{4/(n-2)}\delta_{ij}.$
\end{enumerate}
\end{lemma}

\begin{proof}
For fixed $k,$ let  $v_k$ be the unique solution to the problem
\begin{equation*}
\left\{
\begin{tabular}{rl}
$\Delta v =0$ &  in  $\mathbb{B}_k^c$\\
$v=u$ & on $ \partial \mathbb{B}_k$\\
$v\to 1$ & as  $|x|\to\infty.$
\end{tabular}
\right.
\end{equation*}
Then, it follows that the function 
\begin{equation*}
\tilde u_k(x)=\left\{
\begin{tabular}{rl}
$u(x)$ &  in  $\mathbb{B}_k$\\
$v_k(x)$ & in $ \mathbb{B}_k^c$\\
\end{tabular}
\right.
\end{equation*}
is weakly superharmonic.  We now define
$u_k=\tilde u_k*\phi_{1/k}$, where $\phi_{1/k}$ is a standard mollifier
with support inside $\mathbb{B}_{1/k}$.  It follows that $u_k$ is superharmonic by the
mean value property. (We have used $v_k\le u$ in $\mathbb{B}_k^c$). This
proves (i). (ii) follows from the construction.

To prove (iii), we use that the
  decay estimates for $u$ coming from $g$ being asymptotically flat
  give uniform decay estimates for the $u_k$'s and their derivatives.
  (To obtain this we use the superharmonicity of $u$ and the mean value property).
  Using these uniform estimates it follows  that the surface integrals in the definition
  of mass from equation \eqref{admmass}, evaluated at $g_k=u_k^{4/(n-2)}\delta_{ij}$, 
  converge to the expression evaluated at $g$.  In other workds,  $m_k\to m$ as $k\to\infty$.
\end{proof}

\begin{remark}
For CF-manifolds which are harmonically flat at infinity (thus for any CF-manifold)  the positive mass theorem follows easily for all $n\geq 3$. (Note that  rigidity of the PMT
for CF-manifolds does not follow directly from rigidity in the harmonic-at-infinity case.)
\end{remark}

Indeed,   the transformation law
formula for scalar curvature under conformal deformations  gives that the scalar
curvature of $g_{ij}=u^{\frac 4{n-2}}\delta_{ij}$, denoted by $R_{g}$,
is given by
\begin{equation}\label{sctl}
R_g=u^{-\frac{n+2}{n-2}}\left(-\frac{4(n-1)}{n-2}\Delta_0u+R_{\delta}u\right).
\end{equation}
(Naturally, here $R_{\delta}\equiv 0$.) In particular, we obtain that $R_g\geq 0\Leftrightarrow \Delta_0u\leq 0$. 
The ADM integrand is easily computed in this case:
$$\sum_{i,j}(\partial_jg_{ij}-\partial_ig_{jj})\nu^j=-\frac{4(n-1)}{n-2}u^{\frac{6-n}{n-2}}u_r,$$
and since $u\rightarrow 1$ at infinity, we obtain
$$\int_{S_{\rho}}u^{\frac{6-n}{n-2}}u_rd\sigma^0_{\rho}\sim \int_{S_{\rho}}u_rd\sigma^0_{\rho}=\int_{\mathbb{B}_{\rho}}\Delta_0udx\leq 0.$$
Thus, $m\geq 0$, with equality if and only if $u$ is positive harmonic on $\mathbb{R}^n$ with $u\to 1$ at
infinity, i.e. $u\equiv 1$.

%%%%%%%%%%%%%%%%%%%%%%%%%%%%%%%%%%%%%%%%%%%%%
%%%%%%%%%%%%%%%%%%%%%%%%%%%%%%%%%%%%%%%%%%%%%

\section{Model Case and Main Theorem}

%%%%%%%%%%%%%%%%%%%%%%%%%%%%%%%%%%%%%%%%%%%%%
%%%%%%%%%%%%%%%%%%%%%%%%%%%%%%%%%%%%%%%%%%%%%

The motivation for this note is to investigate whether the mass-capacity inequality holds for 
CF-manifolds  in all dimensions. The following transformation formula for the Laplacian plays a key role.

\begin{lemma} \label{tran} Let $g=u^{\frac4{n-2}}\delta$ and $f\in C^{\infty}(M).$  Then
$\Delta_gf=u^{-\frac{n+2}{n-2}}(\Delta_0(uf)-f\Delta_0u).$  In particular, if $\Delta_0u=0$, then
$ \Delta_0(uf)=0$ if and only if  $\Delta_g f=0.$
\end{lemma}

 We use the Lemma in the following main example.

\begin{ex} Our prototypical example is the so-called Riemannian Schwarzschild manifold. (Compare with
Theorem 9 of \cite{bray01}.) It is constructed
as follows.  For $R_s>0$, denote $m=2R_s^{n-2}$, and define on $\mathbb{B}_{R_s}^c=
\mathbb{R}^n\setminus \mathbb{B}_{R_s}$
the function
\begin{equation}\label{schwarz}
u=1+\left(\frac{R_s}r\right)^{n-2}=1+\frac m{2}r^{2-n}.
\end{equation}
Note that $u$ is actually defined and harmonic
($\Delta_0u=0$) on $\mathbb{R}^n\setminus\{0\}$. \\

Now define
$$\varphi=\frac{1-(\frac{R_s}r)^{n-2}}{1+(\frac{R_s}r)^{n-2}}.$$

Then $\Delta_0(u\varphi)=0$, so $\Delta_g\varphi=0$ by  Lemma \ref{tran} above. Moreover $\varphi\rightarrow 1$ as 
$r\rightarrow \infty$, and $\varphi_{|\Sigma}=0$ for $\Sigma=\partial\mathbb{B}_{R_s}$. Thus, by direct integration
we obtain
$$C_g(\Sigma)=\frac 1{\omega_{n-1}(n-2)}\int_{\mathbb{B}_{R_s}^{c}}|\nabla_g\varphi|_g^2dV_g=\frac 1{n-2}\int_{R_s}^{\infty}u^2\varphi_r^2r^{n-1}dr=m.$$
That is, the equality case of the mass-capacity inequality is achieved by the Riemannian Schwarzschild manifold;
this should be the extremal case for the inequality and it is our motivational starting point. \\
\end{ex}

Introduce the notation: $\alpha=\min_{\Sigma}u$, $\alpha_1=\max_{\Sigma}u$, and $\omega=\alpha_1-\alpha$ is the oscillation of $u$ on $\Sigma$.

%In view of the above example we now generalize the notion of Riemannian Schwarzschild metric to general euclidean domains, 
%not necessarily the 
%complement of a round ball.

%\begin{defn}
%Let $\Omega^c\subset \mathbb{R}^n$. The {\it harmonic} metric of $\Omega^{c}$
%is the asymptotically flat, conformally flat metric $g_s=u_s^{\frac 4{n-2}}\delta$,
%where $u_s>0$ is a positive function on $\Omega^{c}$ which is
%uniquely determined by the following conditions:
%\begin{itemize}
%\item $\Delta_0u_s=0$ (thus $R_{g_s}=0$),
%\item $u_s\to 1$  as $x\to \infty,$
%\item $(u_s)_{\nu}=-\frac{n-2}{2(n-1)}u_sH_0$ on $\Sigma=\partial \Omega$, or equivalently, $H_{g_s}(\Sigma)=0.$
%\end{itemize}
%\end{defn}

%(Note that harmonic metrics are CF-manifolds whenever $\Sigma=\partial\Omega$ is Euclidean mean-convex.)
We are now ready to state our main result.  

\begin{thm}\label{two} Let $n\ge 3$ and $\Omega\subset\mathbb{R}^n$ be a smoothly bounded 
domain with boundary $\Sigma=\partial \Omega$, not necessarily connected. Let $(M,g)$ be isometric
to   a conformally flat metric $g_{ij}=u^{\frac 4{n-2}}\delta_{ij}$ on $\Omega^c$ which is asymptotically 
flat with ADM mass $m$.  (Here    $u>0$ and $u\to1$ towards  infinity.)
Assume  further that $(M,g)$ has non-negative scalar curvature $R_{g}\ge0$.  Then
\begin{itemize}
\item[(I)] If $\Sigma$ is Euclidean mean-convex ($H_0>0$) and $g$-minimal ($H_g=0$), then
$$C_0(\Sigma)< C_g(\Sigma)\leq C_0(\Sigma)+\frac m2.$$
Equality occurs in the  second inequality if and only if $u$ is harmonic.\\

\item[(II)](Euclidean estimate.) Assume (i) $H_0>0$ on $\Sigma$ and (ii) either the solution of inverse mean curvature flow in $\Omega^c$ with initial hypersurface $\Sigma$ is smooth for all $t>0$, or $\Sigma$ is outer-minimizing in $\Omega^c$, or $n<8$. Then: 
$$C_0(\Sigma)\leq \frac 1{(n-1)\omega_{n-1}}\int_{\Sigma}H_0d\sigma_0.$$
Equality holds if and only if $\Sigma$ is a round sphere.\\

\item[(III)] Let $\alpha=\min_{{\Sigma}}u$.
 Under the same assumptions on $\Sigma$ as in (I), we have:
$$\frac 1{(n-1)\omega_{n-1}}\int_{\Sigma}H_0d\sigma_0\leq \frac m{\alpha}.$$
Equality holds if and only if $u$ is harmonic and constant on $\Sigma$ (and, in this case, $\alpha \geq 2$.)(Note that by Lemma \ref{fs} below, $\alpha > 1$ always.)\\

\item[(IV)] Under the same assumptions on $\Sigma$ as in (I) and (II) above, assume $\alpha <2$ and either
$$\alpha\geq 1+\frac{\sqrt 2}2\mbox{ and }\omega <1-\frac{\alpha}2$$
$$\mbox{or }\alpha<1+\frac{\sqrt 2}2\mbox{ and }\omega< 2-\alpha-\frac 1{2\alpha}.$$
Then:
$$C_0(\Sigma)<\frac m2.$$\\

\item[(V)] (Euclidean estimate.) Assume $H_0>0$ on $\Sigma$, and $\Sigma$ is outer-minimizing in $\R^n$ with area $A$. Then: 
$$ \frac1{(n-1)\omega_{n-1}}\int_{\Sigma}H_0d\sigma_0\ge \left(\frac{A}{\omega_{n-1}}\right)^{\frac{n-2}{n-1}}.$$
Equality holds if and only if $\Sigma$ is a round sphere.\\

\end{itemize}

\end{thm}

\begin{remark}
If equality holds in (III), the harmonic function $v=\frac{u-\alpha}{1-\alpha}$ achieving the $inf$ in the definition of $C_0(\Sigma)$ satisfies on $\Sigma$: $v_{\nu}=cH_0, c=(n-2)\alpha/2(n-1)(\alpha-1)$. Whether this can happen in cases other than the sphere is unclear (and is a purely Euclidean question.)
\end{remark}

%%%%%%%%%%%%%%%%%%%%%%%%%%%%%%%%%%%%%%%%%%%%%
%%%%%%%%%%%%%%%%%%%%%%%%%%%%%%%%%%%%%%%%%%%%%

\section{Proof of  Theorem \ref{two}} 

%%%%%%%%%%%%%%%%%%%%%%%%%%%%%%%%%%%%%%%%%%%%%
%%%%%%%%%%%%%%%%%%%%%%%%%%%%%%%%%%%%%%%%%%%%%

Notice that all non-strict inequalities in Theorem \ref{two} are 
closed conditions under
 $C^\infty$ convergence of metrics on $M$.
 Therefore, using the fact that any CF-manifold
 can be approximated, in this topology, by metrics that are
harmonically flat  at infinity (by 
 Lemma \ref{approxi}),  it suffices to prove the non-strict inequalities of
  Theorem \ref{two}
 for metrics that are harmonically flat
 at infinity.  In other words, without loss of generality we may assume that for $(M,g)$ 
with $g=u^{4/(n-2)}\delta_{ij}$,  there exists
 $R_0>0$ large enough so that
$$\Delta_0u=0\mbox{ for }r>R_0,\quad u=1+\frac m{2r^{n-2}}+O\left(\frac 1{r^{n-1}}\right),
\quad u_r=-\frac{(n-2)m}{2r^{n-1}}+O\left(\frac 1{r^n}\right),$$
 where $m$ is the ADM-mass of $(M,g)$.
 
 \begin{remark} For proving the {\bf rigidity} statements as well as
  the strict inequality 
 of (I) in the Theorem we may require additional justifications, as we see below. 
 \end{remark}

The first (strict) inequality of part (I) of Theorem \ref{two} follows from the following lemma of 
\cite{schwartz}.

\begin{lemma}[\cite{schwartz}]\label{fs} Assume $u>0$ and $\Delta_0 u\leq 0$ in $\Omega^c=\mathbb{R}^n\setminus \Omega$, with $\Sigma=\partial \Omega$ mean-convex for the euclidean metric ($H_0>0$) and {\it minimal} for the metric $g=u^{\frac 4{n-2}}\delta$ ($H_g=0$), where $u>0,u\rightarrow 1$ at infinity. Then $u> 1$ on $\Omega^c$. 
\end{lemma}

The key ingredients in the proof of this Lemma are the minimum principle for superharmonic functions and the transformation formula for mean curvature of a hypersurface under conformal deformations of the metric.
This is given by the equation
\begin{equation}\label{mctl}
H_g=u^{-\frac 2{n-2}}\left(H_0+\frac{2(n-1)}{n-2}\frac{u_{\nu}}u\right),
\end{equation}
where $\nu$ is the euclidean-unit outward normal of $\Omega$. (To check the constant multiplying $u_{\nu}/u$, observe that the boundary is minimal for the Riemannian Schwarzschild metric).

%%%%%%%%%%%%%%%%%%%%%%%%%%%%%%%%
%%%%%%%%%%%%%%%%%%%%%%%%%%%%%%%%

\begin{proof}[{\bf Proof of (I)}]
The proof of the first (strict) inequality of (I) is independent of $u$ being
harmonic at infinity.  Indeed, we note that
\begin{equation}\label{inte}
\int_M|\nabla_g\varphi |_g^2dV_g=\int_{\Omega^c}u^{-\frac 4{n-2}}|\nabla_0\varphi|^2u^{\frac{2n}{n-2}}dV_0=\int_{\Omega^c}u^2|\nabla_0\varphi|^2dV_0.
\end{equation}
Since $u> 1$ on $\Omega^c$ from Lemma \ref{fs} above, we conclude that $C_0(\Sigma)\le C_g(\Sigma)$.  
To show that equality is not possible in this inequality, note that both infima for the capacities
are achieved, as discussed in 
Remark 2.
Therefore, if $C_0(\Sigma)= C_g(\Sigma)$, there exists functions $\varphi,\psi$ with equal boundary conditions
so that $C_0(\Sigma)=\int_M|\nabla_0\psi |_0^2dV_0=C_g(\Sigma)=\int_M|\nabla_g\varphi |_g^2dV_g$.  Using equation \eqref{inte} and the
fact that $u>1$ 
we get $\int_{\Omega^c}|\nabla_0\varphi|^2dV_0<\int_{\Omega^c}u^2|\nabla_0\varphi|^2dV_0=\int_M|\nabla_0\psi |_0^2dV_0$,
contradicting the fact that $\psi$ achieves the infimum for the euclidean capacity. (Here we have used the fact that a non-constant 
harmonic function is not constant over sets of positive measure.)\\

Without loss of generality, we prove  the second inequality in (I) under the assumption
that $u$ is harmonic at infinity.
Let $v:\Omega^c\rightarrow (0,1)$ be the unique harmonic function ($\Delta_0v=0$) satisfying $v_{|\Sigma}=0, v(x)\rightarrow 1$ as $x\rightarrow \infty$. Then with $\varphi=\frac vu$ we have
$\varphi_{|\Sigma}=0, \varphi\rightarrow 1$ at infinity. Thus
$$(n-2)\omega_{n-1}C_g(\Sigma)\leq \int_M|\nabla_g\varphi|_g^2dV_g=\int_{\Omega^c}u^2|\nabla_0(\frac vu)|^2dV_0:={\mathcal{I}}=\lim_{\rho\rightarrow \infty}{\mathcal{I}}_{\rho},$$
where
$${\mathcal{I}}_{\rho}:=\int_{\mathbb{B}_{\rho}\setminus \Omega}u^2|\nabla_0(\frac vu)|^2dV_0=\int_{\mathbb{B}_{\rho}\setminus \Omega} \left[|\nabla_0v|^2-\nabla_0(v^2)\cdot \frac{\nabla_0 u}u+v^2\frac{|\nabla_0u|^2}{u^2}\right]dV_0.$$
Now
$$\int_{\mathbb{B}_{\rho}\setminus \Omega}\nabla_0(v^2)\cdot \frac{\nabla_0 u}udV_0=-\int_{\mathbb{B}_{\rho}\setminus \Omega} v^2\mbox{div}_0(\frac{\nabla_0 u}{u})dV_0+\int_{S_{\rho}}v^2\frac{u_r}ud\sigma^0_{\rho}$$
since $v_{|\Sigma}=0$. Noting $\mbox{div}_0(u^{-1}\nabla_0 u)=u^{-1}\Delta_0 u-u^{-2}|\nabla_0u|^2$ and $\Delta_0 u\leq 0$, we have
\begin{equation}\label{equ}
\int_{\mathbb{B}_{\rho}\setminus \Omega}\nabla_0(v^2)\cdot \frac{\nabla_0 u}udV_0\geq \int_{\mathbb{B}_{\rho}\setminus \Omega}v^2\frac{|\nabla_0u|^2}{u^2}dV_0+\int_{S_{\rho}}v^2\frac{u_r}ud\sigma_{\rho}^0,
\end{equation}
and hence we obtain
$${\mathcal{I}}_{\rho}\leq \int_{\mathbb{B}_{\rho}\setminus \Omega}|\nabla_0v|^2dV_0-\int_{S_{\rho}}v^2\frac{u_r}ud\sigma_{\rho}^0.$$
Taking limits as $\rho\rightarrow \infty$ and using the asymptotics of $u$ we find
$${\mathcal{I}}\leq (n-2)\omega_{n-1}C_0(\Sigma)+(n-2)\omega_{n-1}\frac m2.$$
From this it follows $C_g(\Sigma)\leq C_0(\Sigma)+\frac m2$, as claimed.\\

\noindent
{\bf Rigidity of (I).}  Here we do not assume that $u$ is harmonic at infinity.\\

\noindent
{\bf Claim.} Equality in the second inequality of (I) implies that
$u$ is harmonic.

\begin{proof}
We proceed by contradiction. Assume that  $C_g(\Sigma)=C_0(\Sigma)+\frac m2$, 
but $u$ is {\it not} harmonic.  Then,
   there exists $p\in\R^n$ and $a,b>0$ so that
$\Delta u\le-a< 0$ in $\mathbb B_b(p)$. (Recall that $u$ is superharmonic, i.e. $\Delta u\le 0$.)
Let $\{u_k\}$ be the approximation of $u$ by functions that are harmonic at infinity
as in Lemma \ref{approxi}.   
  For $\rho,k\ge |p|+b$, equation \eqref{equ} above
   (written for $u_k$) may be replaced by
\begin{align*}
\int_{\mathbb{B}_{\rho}\setminus \Omega}\nabla_0(v^2)\cdot \frac{\nabla_0 u_k}{u_k}
dV_0
 \geq & \int_{\mathbb{B}_{\rho}\setminus \Omega}v^2\frac{|\nabla_0u_k|^2}{u_k^2}dV_0+\int_{S_{\rho}}v^2\frac{(u_k)_r}{u_k}d\sigma_{\rho}^0\\
 &-\int_{B_b(p)} v^2\frac{\Delta_0u}{u}dV_0\notag \\
 \ge&   \int_{\mathbb{B}_{\rho}\setminus \Omega}v^2\frac{|\nabla_0u_k|^2}{u_k^2}dV_0+\int_{S_{\rho}}v^2\frac{(u_k)_r}{u_k}d\sigma_{\rho}^0+Cab^n,\notag
\end{align*}
where $v$ is as in equation \eqref{equ} and $C>0$ some positive constant
that depends on $u,v$ and $n$. From this it follows that

$$\mathcal{I}(u_k)\le (n-2)C_0(\Sigma)+(n-2)\omega_{n-1}\frac{m_k}{2}-Cab^n.$$
Taking limit $k\to\infty$ above gives $C_g(\Sigma)\leq C_0(\Sigma)+\frac m2
-Cab^n<C_0(\Sigma)+\frac m2.$ This contradicts the fact that 
$C_g(\Sigma)=C_0(\Sigma)+\frac m2$.  We deduce that $u$ must be harmonic, 
and one direction of the rigidity statement follows.  
\end{proof}

\noindent
{\bf Claim.}
If $u$ is harmonic (and therefore harmonic at infinity), the second inequality of (I) is an
equality.

\begin{proof}
Let $\psi$ be the function  that achieves the infimum for the capacity
$C_g(\Sigma)$.  From Lemma \ref{tran} it follows that such function (i.e. the $g$-harmonic function which is
exactly zero on $\Sigma$ and goes to one at infinity) satisfies $\Delta_{0}(u\psi)=0$.  We immediately 
recognize that $u\psi$ must be equal to the function $v$ from above since both are harmonic
and equal on $\Sigma$ and at infinity.  
Therefore, $\psi$ equals $\varphi=v/u$
from above, and all the above inequalities become equalities. 
\end{proof}

\noindent
 This finishes the proof of (I). 
\end{proof}

%%%%%%%%%%%%%%%%%%%%%%%%%%%%%%%%
%%%%%%%%%%%%%%%%%%%%%%%%%%%%%%%%

\begin{proof}[{\bf Proof of (II)}]  (This estimate is purely Euclidean;
no approximation by metrics that are harmonically flat at infinity is needed.)
Here we use a modification of the method described in \cite{braymiao}.  
First, we get an upper bound for $C_0(\Sigma)$ 
 using test functions of the form $\varphi=f\circ \phi$, where $\phi \in C^1(\Omega^c,\mathbb{R}_+)$ 
 is a (soon to be determined) proper function vanishing on $\Sigma=\Sigma_0$  whose level sets define a foliation 
 $(\Sigma_t)_{t\geq 0}$ of $\Omega^c$.  As noted in \cite{braymiao}, we have
\begin{equation}\label{vari}
(n-2)\omega_{n-1}C_0(\Sigma)\leq \inf \left\{\int_0^{\infty}(f')^2w(t)dt~:~f(0)=0,f(\infty)=1\right\},
\end{equation}
where $w(t)=\int_{\Sigma_t}|\nabla_0\phi| d\sigma^0_t>0.$\\

(We omit the subscript/superscript `0' for the remainder of the proof of (II).)\\

 Moving away from the method of  \cite{braymiao}, we note that the one-dimensional variational problem 
 \eqref{vari} is easily solved. 
\begin{claim} 
Provided  $w^{-1}\in L^1(0,\infty)$, the infimum of the right hand side of \eqref{vari} equals
$\mathbb{I}^{-1}=(\int_0^{\infty}\frac 1{w(s)}ds)^{-1}$, and is attained by the function 
$f(t)=\frac 1{\mathbb{I}}\int_0^tw^{-1}(s)ds.$   
\end{claim}

\begin{proof}
This follows from
$$1=\int_0^{\infty}f'dt=\int_0^{\infty}f'w^{1/2}w^{-1/2}dt\leq \left(\int_0^{\infty}(f')^2w(t)dt\right)^{1/2}
\left(\int_0^{\infty}w^{-1}(t)dt\right)^{1/2}.$$\end{proof}

\begin{remark} If $\Omega\subset \mathbb{R}^n$ is {\it convex}, it is natural to try to use the distance
function $\phi=dist(\cdot,\Sigma)$ for the above process.  In this case, the level sets of $\phi$ 
give a  foliation of $\Omega^c$ by outer parallel hypersurfaces.  We get  $|\nabla \phi|\equiv 1$, so 
$w(t)=|\Sigma_t|$  is the Euclidean (n-1)-dimensional area. By a well-known formula
$$|\Sigma_t|=|\Sigma|+\sum_{j=0}^{n-2}\left(\int_{\Sigma}\sigma_j(\vec{k})d\sigma\right)t^j+\omega_{n-1}t^{n-1},$$
where $\sigma_j(\vec{k})$ is the j-th elementary symmetric function of the principal curvatures 
$\vec{k}=(k_1,\ldots, k_{n-1}), ~k_i>0$ of $\Sigma$.  Now since 
$$\sigma_1(\vec{k})=H\mbox{ and }  \sigma_j(\vec{k})\leq H^j\mbox{ for }j=1,\dots,n-1,$$
we see that an estimate based on this foliation would involve the integrals $\int_{\Sigma}H^jd\sigma$. 
Since we are interested in estimating the capacity in terms of the ADM mass (especially in view of part (III)), 
 we  choose a different function to construct the foliation.\\
\end{remark}

\noindent
Consider the foliation $({\Sigma_t})_{t\geq 0}$ defined by the level sets of the function 
given by Huisken and Ilmanen's weak solution of inverse mean curvature flow 
\cites{huiskenilmanen,hi08} in $\Omega^c\subset \mathbb{R}^n$. 
We recall the summary given in \cite{braymiao} (which holds in all dimensions):

\begin{thm}[Huisken-Ilmanen, \cites{huiskenilmanen,hi08}] \label{huil} \noindent
\begin{itemize}
\item There exists a proper, locally Lipschitz function $\phi\geq 0$ on $\Omega^c$, $\phi_{|\Sigma}=0$. For $t>0$, $\Sigma_t=\partial\{\phi < t\}$ and 
$\Sigma_t'=\partial \{\phi>t\}$ define increasing families of $C^{1,\alpha}$ hypersurfaces;

\item The hypersurfaces $\Sigma_t$ (resp.$\Sigma_t'$) minimize (resp. strictly minimize) area among surfaces homologous to $\Sigma_t$ in $\{\phi\geq t\}\subset \Omega^c$. The hypersurface $\Sigma'=\partial \{\phi>0\}$ strictly minimizes area among hypersurfaces homologous to $\Sigma$ in $\Omega^c$.

\item There exists a closed singular set $Z\subset \Omega^c$ not intersecting $\Sigma$, of Hausdorff codimension at least 8 in $\mathbb{R}^n$, so that if $\Sigma_t$ (resp. $\Sigma_t'$) does not intersect $Z$ we have (in the sense of $C^{1,\beta} convergence,\beta<\alpha$):
$$\Sigma_s\rightarrow \Sigma_t\mbox{ as }s\uparrow t; \mbox{ (resp.) }\Sigma_s\rightarrow \Sigma_t'\mbox{ as }s\downarrow t.$$

\item For almost all $t>0$, the weak mean curvature of $\Sigma_t$ is defined and equals $|\nabla \phi|$, which is positive a.e. on $\Sigma_t$. 
\end{itemize}
\end{thm}

From Theorem \ref{huil} and  the Claim from above it follows that
\begin{equation}\label{imcf2}
(n-2)\omega_{n-1}C_0(\Sigma)\leq \left(\int_0^{\infty}w^{-1}(t)dt\right)^{-1}, \mbox{ where } w(t):=\int_{\Sigma_t}Hd\sigma_t.\end{equation}

\begin{lemma} \label{imcf} Consider the foliation $\{\Sigma_{t}\}$ given by IMCF in 
$\Omega^c\subset \mathbb{R}^n$ as above. Then
$$\int_{\Sigma_t}Hd\sigma\leq \left(\int_{\Sigma_0}Hd\sigma\right)e^{\frac{n-2}{n-1}\cdot t}$$
 for $t\geq 0$. \end{lemma}

\begin{remark}\label{sphere}
Note that equality holds in the above inequality 
 for the foliation by  IMCF outside a sphere, which is  given by $\Sigma_t=\partial\mathbb{B}_{R(t)}\subset \mathbb{R}^n$,
 where  $R(t)=e^\frac t{n-1}$.
\end{remark}

\begin{proof}[Proof of Lemma \ref{imcf}] From \cite{huiskenilmanen} we have that, so long as the 
evolution remains smooth, 
\begin{equation}\label{aleks}
\frac d{dt}\left(\int_{\Sigma_t}Hd\sigma_t\right)=\int_{\Sigma_{t}}\left(H-\frac{|A|^2}H\right)d\sigma_t\leq \frac{n-2}{n-1}\int_{\Sigma_t}Hd\sigma_t,
\end{equation}
 where $A$ denotes the second fundamental form, and   the second inequality follows from
\begin{equation}\label{ric}
H-\frac{|A|^2}H-\frac{n-2}{n-1}H=\frac 1{(n-1)H}(H^2-(n-1)|A|^2)\leq 0.
\end{equation}
(Note that equality occurs in this last inequality if and only if each connected component of $\Sigma_t$ is a sphere.)  This concludes the proof in the case of smooth solutions.
\end{proof}

\begin{proof}[Proof of Lemma \ref{imcf} in the case of weak solutions]

We need the assumption $\Sigma'\cap Z=\emptyset$: the outer minimizing hull of $\Sigma$ does not intersect the singular set of $\phi$. This is true in case $n<8$ (then $Z=\emptyset$) or if $\Sigma$ is outer-minimizing (then $\Sigma'=\Sigma$, which does not intersect $Z$.)\vspace{.2cm}

By \cite{huiskenilmanen} a variational solution of IMCF $\phi\in Lip_{loc}(\Omega^c;\mathbb{R}_+)$ (proper) can be approximated (locally in Lipschitz norm) by proper functions $u_i\in C^2(\Omega^c)$ with $L^2$ convergence of the (weak) mean curvature of level sets. For each $t>0$, with $\Omega_t=\{\phi < t\}$:
$$\phi_i\rightarrow \phi\mbox{ in }{ Lip}(\Omega_t),\quad H_i\rightarrow H\mbox{ in }L^2(\Omega_t).$$

From Lemma A1 in the Appendix, for any $\Phi\in Lip(0,t)$, $\Phi\geq 0$ with compact support and $\varphi_i=\Phi\circ \phi_i,$ we have:
$$-\int_{\Omega_t}\nabla \varphi_i\cdot \nu_iH_idV_0=\int_{\Omega_t}\varphi_i(H_i^2-|A_i|^2)dV_0
\leq \frac{n-2}{n-1}\int_{\Omega_t}\varphi_iH_i^2dV_0.$$

Taking limits as $i\rightarrow \infty$, since $\varphi_i\rightarrow \varphi:=\Phi\circ \phi$ in $Lip(\Omega_t)$, we obtain:
$$-\int_{\Omega_t}\nabla \varphi\cdot \nu HdV_0\leq \frac{n-2}{n-1}\int_{\Omega_t}\varphi H^2dV_0.$$
\vspace{.3cm}

Now given $0<\bar{t}<t$ and $0<\delta < (t-\bar{t})/2$,  define $\Phi$ on $[0,t]$ by:
$$\Phi(s)=0\mbox{ on }[0,\bar{t}];\quad\Phi(s)=(s-\bar{t})/\delta\mbox{ on }[\bar{t},\bar{t}+\delta];\quad \Phi(s)=1\mbox{ on }[\bar{t}+\delta,t-\delta];\quad \Phi(s)=(t-s)/\delta\mbox{ on }[t-\delta,t].$$
Let $\varphi=\Phi\circ \phi$. Using the inequality just derived, the coarea formula and the fact that $H=|\nabla \phi|$ a.e.:
$$-\int_0^t\Phi'(s)(\int_{\Sigma_s}Hd\sigma_s)ds=-\int_{\Omega_t}(\Phi'\circ \phi)\nabla \phi\cdot \nu HdV_0=-\int_{\Omega_t}\nabla \varphi \cdot \nu HdV_0$$
$$\leq \frac{n-2}{n-1}\int_{\Omega_t}\varphi H^2dV_0
=\frac{n-2}{n-1}\int_{\Omega_t}\varphi H|\nabla \phi|dV_0
=\frac{n-2}{n-1}\int_0^t\Phi(s)\int_{\Sigma_s}Hd\sigma_sds.$$
Since the left hand side equals $(1/\delta)[\int_{t-\delta}^{\delta}\int_{\Sigma_s}Hd\sigma_sds-\int_{\bar{t}}^{\bar{t}+\delta}\int_{\Sigma_s}Hd\sigma_sds]$,
we find, letting $\delta\rightarrow 0$, for a.e. pair $0<\bar{t}<t$ :
$$\int_{\Sigma_t}Hd\sigma_t\leq\int_{\Sigma_{\bar{t}}}Hd\sigma_{\bar{t}}
+ \frac{n-2}{n-1}\int_{\bar{t}}^t\int_{\Sigma_s}Hd\sigma_sds.$$
By assumption, the singular set $Z$ does not intersect $\Sigma'$. Thus we may let $\bar{t}_i\downarrow 0$ and conclude (see (1.13) in \cite{huiskenilmanen}):
$$\Sigma_{\bar{t}_i}\rightarrow \Sigma', \quad \int_{\Sigma_{\bar{t}_i}}Hd\sigma_{\bar{t}_i}\rightarrow \int_{\Sigma'} Hd\sigma'.$$
On the other hand, (1.15) in \cite{huiskenilmanen}:
$$H_{\Sigma'}=0\mbox{ on }\Sigma'\setminus \Sigma; \quad H_{\Sigma'}=H_{\Sigma}\mbox{ a.e. on }\Sigma'\cap \Sigma, \quad |\Sigma|=|\Sigma'|$$
imply:
$$\int_{\Sigma'}Hd\sigma'\leq \int_{\Sigma}Hd\sigma.$$
We conclude that for a.e. $t>0$:
$$\int_{\Sigma_t}Hd\sigma_t\leq\int_{\Sigma}Hd\sigma
+ \frac{n-2}{n-1}\int_0^t\int_{\Sigma_s}Hd\sigma_sds.$$

Now the claim of Lemma \ref{imcf} follows from Gronwall's Lemma. 

\end{proof}

By straightforward integration, Lemma \ref{imcf} implies:
$$\left(\int_0^{\infty}w^{-1}(t)dt\right)^{-1}\leq \frac{n-2}{n-1}\int_{\Sigma}Hd\sigma.$$
Together with equation \eqref{imcf2} this gives
$$C_0(\Sigma)\leq \frac 1{(n-1)\omega_{n-1}}\int_{\Sigma}Hd\sigma,$$
as claimed in part (II) of the main theorem.\\

\noindent
{\bf Rigidity of (II).}  From Remark \ref{sphere} it follows that the  inequality of part (II) is an equality whenever
$\Sigma$ is a round sphere.  

%On the other hand, if  equality holds in part (II), it follows that 
%the evolution by IMCF does not have any jumps, i.e. remains classical.  
%Indeed, if the flowing hypersurface
%jumps,  then the total mean curvature {\it strictly
%decreases} --this follows from arguments in the solution of the obstacle problem.  In other words, if  the 
%flow were to jump, the inequality in (II) would be is strict.  Therefore, in the case of equality
% there are no jumps.  In particular this implies that $\Sigma$ is connected.

On the other hand, if equality holds in part (II), it follows that 
$$\int_{\Sigma_t}Hd\sigma= \left(\int_{\Sigma_0}Hd\sigma\right)e^{\frac{n-2}{n-1}\cdot t}\mbox{ for a.e. } t\geq 0,$$
and therefore:
$$H^2=(n-1)|A|^2\mbox{ on }\Sigma_t,\mbox{ for a.e. }t\geq 0.$$
This implies $\Sigma_t$ is a disjoint union of  round spheres, for a.e. $t\geq 0$. For a solution of inverse mean curvature flow in $\mathbb{R}^n$, this is  only possible if $\Sigma_t$ is, in fact, a single round sphere for every $t$. 
(See e.g. the Two Spheres Example 1.5 of \cite{huiskenilmanen}.) This proves part (II).\end{proof}

%Finally,  equality in equation \eqref{ric} gives that
 %$\Sigma_{t}$ is a round sphere for all $t$.  This proves part (II).  \end{proof}

%%%%%%%%%%%%%%%%%%%%%%%%%%%%%%%%
%%%%%%%%%%%%%%%%%%%%%%%%%%%%%%%%

\begin{proof}[{\bf Proof of (III)}]  Here we may assume, by the argument described at
the beginning of \S 4, that
$u$ is harmonic at infinity.
 From the transformation law for the mean curvature
 given by equation \eqref{mctl}, together with the divergence theorem, it follows that
\begin{align*}
 \int_{\mathbb{B}_{\rho}\setminus \Omega}{\Delta_0u}dV_0&= \int_{S_{\rho}}{u_r}d\sigma_{\rho}^0-\int_{\Sigma}{u_{\nu}}d\sigma_0\\
 &= -m\omega_{n-1}\frac{n-2}{2}+O(\rho^{-1})+\frac{n-2}{2(n-1)}\int_{\Sigma}H_{0}ud\sigma_0.
\end{align*}
Taking the limit $\rho\to\infty$ we obtain
\begin{equation}\label{mass}
m=-\frac 2{(n-2)\omega_{n-1}}\int_{\Omega^c}{\Delta_0u}dV_0+\frac{1}{(n-1)\omega_{n-1}}
\int_{\Sigma}H_{0}ud\sigma_0.
\end{equation}
Since $\Delta_0u\leq 0$ on $\Omega^c$ and $u\geq \alpha$ on $\Sigma$, this gives the inequality in (III).\\

\noindent
{\bf Rigidity of (III).}  
 For the rigidity statement of (III) we only need to prove one direction 
 since (clearly) for the Riemannian Schwarzschild manifold, the above inequalities are all equalities. Here we may not assume that $u$ is harmonic at infinity (although this will
 follow from the claim below).

If equality holds in (III), we have that
 \begin{equation}\label{rigidity3a}
 \int_{\Sigma}H_0d\sigma_0=(n-1)\omega_{n-1}\frac m{\alpha}.
 \end{equation}
    
 \begin{claim} $u$ is harmonic on $\Omega^c$, and is (the same) constant on (all components of) $\Sigma$.
  \end{claim} 
  
\begin{proof} Let $\{u_k\}$ be the approximating sequence of $u$ from Lemma \ref{approxi}. 
If $\Delta u\neq 0$ at some point in $\Omega^c$, then $\frac 2{(n-2)\omega_{n-1}}\int_{\Omega^c}{\Delta_0u}dV_0<0.$
Writing  equation \eqref{mass} for $u_k$ and taking the limit $k\to\infty$
contradicts equation \eqref{rigidity3a}.  Thus, $u$ must be harmonic, and therefore
harmonic at infinity.
Note that since the inequality in (III) is obtained from 
equation  \eqref{mass} by replacing $u$  by its 
minimum on 
$\Sigma$, it follows that, in the case of equality in (III),   $u$ equals its minimum  
on $\Sigma$, i.e.
 $u|_{\Sigma}\equiv \min_{\Sigma}u=\alpha.$   \end{proof}

\begin{claim} $\alpha \geq 2$.
\end{claim}

\begin{proof} From the previous claim $\Delta_0u=0$, so
$$0=\int_{\Omega^c}u\Delta_0udV_0=-\int_{\Omega^c}|\nabla_0 u|^2dV_0-\frac m2 \omega_{n-1}(n-2)-\int_{\Sigma}uu_{\nu}d\sigma_0.$$
Also, from that claim  $u_{|\Sigma}\equiv \alpha$, so we know $u$ is the optimal function for
 $C_0^{(\alpha, 1)}(\Sigma)$ (cf. Remark 2).  Furthermore, using Remark 3
 it follows that  $\int_{\Omega^c}|\nabla_0 u|^2dV_0=(n-2)\omega_{n-1}(\alpha-1)^2C_0(\Sigma).$
 Combining this with the above equation we obtain
 
\begin{equation*}
(n-2)\omega_{n-1}(\alpha-1)^2C_0(\Sigma)=-\frac m2\omega_{n-1}(n-2)+\frac{n-2}{2(n-1)}\alpha^2\int_{\Sigma}H_0d\sigma_0.
\end{equation*}

We now use equation \eqref{rigidity3a} to substitute the last term in the above equation.  We get
\begin{equation}\label{rigidity3}
(\alpha -1)C_0(\Sigma)=\frac m2.
\end{equation}

It is easy to see that equations \eqref{rigidity3a}, \eqref{rigidity3},
combined with the inequality in (II),  imply $\alpha \geq 2$. 
\end{proof}
This concludes the proof of (III). 
\end{proof}

%%%%%%%%%%%%%%%%%%%%%%%%%%%%%%%%
%%%%%%%%%%%%%%%%%%%%%%%%%%%%%%%%

\begin{proof}[{\bf Proof of (IV)}]  
Here we may assume, once again by the argument described at
the beginning of \S 4, that
$u$ is harmonic at infinity.  

Let $f:[1,\infty)\rightarrow \mathbb{R}_+$ be a $C^2$ function 
satisfying the following conditions: 
\begin{align}\label{eff}
\left\{
\begin{tabular}{l}
$f> 0, ~f'< 0,~f''>0$ on $[1,\infty)$, 
$f(1)=1$, and $f'(1)=-1$ \\
\end{tabular}
\right.
\end{align}

Then $f\circ u\rightarrow 1$ at infinity, while $0<f(\alpha_1)\leq (f\circ u)_{|\Sigma}\leq f(\alpha)$.

Note that since
\begin{equation}\label{lap}
\Delta_0(f\circ u)=(f''\circ u)|\nabla_0u|^2+(f'\circ u)\Delta_0 u\geq 0,
\end{equation}
we have
\begin{align} \label{3s}
0\leq\int_{\mathbb{B}_{\rho}\setminus \Omega}|\nabla_0(f\circ u)|^2dV_0+&\int_{B_{\rho}}(f\circ u)(f''\circ u)|\nabla_0u|^2dV_0-
\int_{B_{\rho}}(f\circ u)|f'\circ u|\Delta_0 udV_0\\
\notag&=
\int_{S_{\rho}}f\circ u(f\circ u)_rd\sigma^0_{\rho}-\int_{\Sigma}f\circ u(f\circ u)_{\nu}d\sigma.
\end{align}
For the integral over $\Sigma$, using the 
boundary condition on $u$ we obtain
$$\int_{\Sigma}f\circ u(f\circ u)_{\nu}d\sigma=\int_{\Sigma}f\circ u(f'\circ u)u_{\nu}d\sigma=\frac{2(n-2)}{n-1}\int_{\Sigma}f\circ u|f'\circ u|uH_0d\sigma.$$
For the first term of the right hand side of equation \eqref{3s} we use the asymptotics of $u_r$ to see that, in 
 the limit $\rho\rightarrow\infty$,  the integral over $S_{\rho}$ satisfies
$$\lim_{\rho \rightarrow \infty}\int_{S_{\rho}}f\circ u(f\circ u)_rd\sigma^0_{\rho}=-\frac m2(n-2)\omega_{n-1}f(1)f'(1).$$
Thus:
$$\frac{2\alpha}{n-1}f(\alpha_1)|f'(\alpha_1)|\int_{\Sigma}H_0d\sigma\leq 
\frac{2}{n-1}\int_{\Sigma}f\circ u|f'\circ u|uH_0d\sigma\leq \frac m2\omega_{n-1}f(1)|f'(1)|.$$
By Theorem 5, part II:
$$C_0(\Sigma)\leq \frac 1{(n-1)\omega_{n-1}}\int_{\Sigma}H_0d\sigma_0.$$
Combining these two facts, we have:
$$2\alpha f(\alpha_1)|f'(\alpha_1)|C_0(\Sigma)\leq \frac m2f(1)|f'(1)|=\frac m2.$$
To finish the proof we need:
$$2\alpha f(\alpha_1)|f'(\alpha_1)|:=\mu>1.$$
We {\it claim} that, under the assumptions on the oscillation $\omega$, it is possible to find $f$
as in (14), satisfying also this condition.\\

{\it Proof of claim.} This is elementary. Note (i) $2-\alpha-\frac 1{2\alpha}>0$ iff $1-\sqrt{2}/2<\alpha<1+\sqrt{2}/2$; (ii) $1-\alpha/2>2-\alpha-\frac 1{2\alpha}$ for $\alpha>1$.\\

The convexity constraint (necessary and sufficient for existence of convex $f$) is:
$$|f'(\alpha_1)|<\lambda:=\frac{1-f(\alpha_1)}{\beta_1}<1,\mbox{ where }\beta_1=\alpha_1-1>0.$$
We must have, for some $\mu>1$ and $\lambda>0$:
$$\lambda <1, \quad |f'(\alpha_1)|=[2\alpha(1-\lambda\beta_1)]^{-1}\mu<\lambda.$$
Thus we need $\lambda$ to satisfy:
$$\lambda<1, \quad 2\alpha\lambda(1-\lambda\beta_1)>1.$$
The second inequality is equivalent to:
$$p(\lambda):=2\alpha \beta_1\lambda^2-2\alpha \lambda+1<0.$$
The discriminant of $p$ is $\Delta=4\alpha(\alpha-2\alpha_1+2)>0$,
from the hypothesis on the oscillation: $\omega=\alpha_1-\alpha<1-\frac{\alpha}2$.
The roots of $p(\lambda)$ are:
$$r_{\pm}=\frac 1{2\beta_1}(1\pm\sqrt{1-\frac{2\beta_1}{\alpha}}).$$
We need $r_{-}<1$, or equivalently:
$$1-2\beta_1<\sqrt{1-2\beta_1/\alpha}.$$
If $\alpha\geq 1+\sqrt{2}/2$, then certainly $\alpha_1>3/2$, or $\beta_1>1/2$ and we are done.\\

Otherwise, the condition needed is  equivalent to $2\alpha(1-\beta_1)>1$, or $\omega<2-\alpha-\frac 1{2\alpha}$. This concludes the proof of the claim.
\end{proof}

%%%%%%%%%%%%%%%%%%%%%%%%%%%%%%%%
%%%%%%%%%%%%%%%%%%%%%%%%%%%%%%%%

\begin{proof}[{\bf Proof of (V)}]  (This estimate is purely Euclidean, so
no approximation by metrics that are harmonically flat at infinity is required.)
Recall that a hypersurface $\Sigma=\partial \Omega\subset \R^n$ is called 
{\bf outer-minimizing}
if whenever $\Omega'$ is a domain with $\Omega'\supset\Omega$ then $|\partial \Omega'|\ge|\Sigma|$.  (An
example of such a hypersurface is given by the boundary of a collection of sufficiently far-apart convex bodies in $\mathbb{R}^n$.)  Let us denote by 
$|\Sigma_{t}|$ the area of the evolving hypersurface $\Sigma_{t}$ moving
 by IMCF with initial condition $\Sigma_{0}\equiv \Sigma$.  Then, by Lemma 1.4 of \cite{huiskenilmanen}, one 
 has $|\Sigma_t|=e^t|\Sigma|$ for all $t\geq 0$, provided $\Sigma$ is outer-minimizing.

Now, from Lemma \ref{imcf} and the fact that $e^{(\frac{n-2}{n-1})t}=(|\Sigma_t|/|\Sigma|)^{\frac{n-2}{n-1}}$, 
we have that the function
$$f(t):=|\Sigma_{t}|^{-\frac{n-2}{n-1}}\int_{\Sigma_{t}}Hd\sigma_{t}$$
is non-increasing along IMCF in $\mathbb{R}^n$.  By a known property of Euclidean IMCF, for $t$ large enough
 $\Sigma_{t}$ is arbitrarily close to a round sphere, and hence 
$f(t)\to(n-1)\omega_{n-1}^{1/(n-1)}$ as $t\to\infty$. This proves the inequality in (V), since
 $f(0)=|\Sigma|^{-(n-2)/(n-1)}\int_{\Sigma}Hd\sigma$.\\

\noindent
{\bf Rigidity of (V).}  From Remark \ref{sphere} it follows that the  inequality of part (V) is an equality whenever
$\Sigma$ is a round sphere.  On the other hand, if the inequality in (V) were an equality, we have
$f(\infty)= f(0),$ so $f(t)\equiv f(0)$
for all $t$ since $f$ is non-increasing.  This implies $\int_{\Sigma_{t}}Hd\sigma_{t}=ce^{t(n-2)/(n-1)},$  and
 inequality  \eqref{aleks}
becomes an equality.  Thus, we have reduced rigidity here to the case of rigidity of part (II).
\end{proof}

%%%%%%%%%%%%%%%%%%%%%%%%%%%%%%%%%%%%%%%%%%%%%
%%%%%%%%%%%%%%%%%%%%%%%%%%%%%%%%%%%%%%%%%%%%%

\section{Applications of the Main Theorem}

%%%%%%%%%%%%%%%%%%%%%%%%%%%%%%%%%%%%%%%%%%%%%
%%%%%%%%%%%%%%%%%%%%%%%%%%%%%%%%%%%%%%%%%%%%%

\begin{proof}[{\bf Proof of part (a) of Theorem \ref{main}}]  The inequality $C_g(\Sigma)\leq m$ follows immediately 
combining parts (I) , (II) and (III) (in case (i)), or parts (I) and (IV) (in case (ii)) of Theorem 5. \\

\end{proof}

\begin{proof}[{\bf Proof of part (b) of Theorem \ref{sch}}]  As observed in \cite{schwartz}, by spherical decreasing rearrangement
  the Euclidean capacity of $\partial \Omega$ is bounded from below by the capacity of a ball with the same 
  volume as $\Omega$. Namely, the ball of radius 
  $R=(V/\beta_n)^{1/n},\beta_n=vol_0(\mathbb{B}^n),V_0=vol_0(\Omega)$.  In other words, 
\begin{equation}\label{cap}
C_0(\Sigma)\geq \left(\frac{V_0}{\beta_n}\right)^{\frac{n-2}n}.
\end{equation}
On the other hand, part (a) of Theorem \ref{two} gives that $m\ge 2 C_0(\Sigma)$ (with strict inequality in case (ii)).  Together with \eqref{cap} this gives
$m\geq 2\left({V_0}/{\beta_n}\right)^{\frac{n-2}n},$
which is the claim of part (b) of Theorem \ref{sch}. \\ 

\end{proof}

\begin{proof}[{\bf Proof of part (c) of Theorem \ref{main}}]  Equality in part (a) implies we must have $\alpha=2$ and equality in Theorem 5(II), 5(III) and the second inequality of 5(I). Thus $\Sigma$ is a round sphere, $u\equiv 2$ on $\Sigma$ and $u$ is harmonic. It follows $g$ is Riemannian Schwarzschild.

In the same way, equality in part (b) implies $\alpha \geq 2$ and $C_0(\Sigma)=m/2$. By the case of equality in Theorem 5(II) and (III), it follows again that $\alpha=2$, $\Sigma$ is a sphere and $u$ is harmonic; hence $g$ is Riemannian Schwarzschild.\\

\end{proof}

\begin{proof}[{\bf Proof of Theorem \ref{afr}}] This is just parts (II) and (V) of Theorem
 \ref{two}. 
\end{proof}

\begin{remark} Combining theorem 1(b) and theorem 2(a), we find:
$$\frac 1{(n-1)\omega_{n-1}}\int_{\Sigma}H_0d\sigma_0\geq C_0(\Sigma)\geq \left(\frac{V_0}{\beta_n}\right)^{\frac{n-2}n}.$$
The resulting inequality between total mean curvature and volume is weaker than theorem 2(b) (via the isoperimetric inequality in $\mathbb{R}^n$). 
\end{remark}

\begin{remark} 
 The proof of the Riemannian Penrose inequality in \cite{bray01} involves 
the construction of a conformal flow of asymptotically flat Riemannian metrics $(g(t))_{t\geq 0}$ in the conformal 
class of the initial metric $g(0)$. It is crucial for the argument in \cite{bray01} that the ADM mass $m(t)$ of $g(t)$ be 
non-increasing in $t$. The proof of this is based on the relation (\cite{bray01}, section 7):
$$\frac d{dt}m(t)_{|t=t_0}=C_{g_{t_0}}(\Sigma(t_0))-m(t_0)$$
(using the normalization in the present paper for the capacity, and one-sided derivatives at the ``jump times''). 
Given this relation, the fact that $m(t)$ is non-increasing follows from the mass-capacity inequality obtained here
 for conformally flat metrics (independently of the positive mass theorem, or PMT), while in \cite{bray01} (for more
  general metrics, in dimension 3) it is obtained applying the reflection argument of \cite{masood} and the PMT. 
  (In fact, this is apparently the only place in \cite{bray01} where the PMT is needed.) Thus our  result of part (a)
  may 
  be regarded as evidence that the Riemannian Penrose inequality for conformally flat metrics in all dimensions 
  can be obtained from arguments of classical linear elliptic theory, 
as conjectured by Bray and Iga in \cite{brayiga}.\end{remark}

\begin{remark}  In his recent proof of the Penrose inequality for asymptotically flat graphs of functions,
M-K. G. Lam \cite{Lam11} uses the Aleksandrov-Fenchel inequality (Theorem 2(b)).  Our theorem can be used
{\it as-is} to strengthen Lam's result.  This is obtained by replacing 
``convex boundary components'' by  ``Euclidean mean-convex, outer minimizing boundary components'' in his
proof of the inequality.  (Or more precisely, by replacing Lemma 12  of \cite{Lam11} with Theorem 2(b).)
We obtain the following result: If $(M,g)$ is the graph of a smooth, asymptotically flat function 
$f:\mathbb{R}^n\setminus \Omega \rightarrow \mathbb{R}$ (as in \cite{Lam11}), and $\Sigma= f^{-1}(0)$ 
consists of  
Euclidean mean-convex, outer minimizing boundary components, then
$$m\geq \frac12\sum_{i=1}^k\left(\frac{|\Sigma_i|}{\omega_{n-1}}\right)^{\frac{n-2}{n-1}}.$$
\end{remark}

%\noindent
%{\bf Application of Theorem \ref{two}.} 
%Combining (I), (II) and (III) from Theorem \ref{two}  we conclude:
%$$C_g(\Sigma)\leq \frac m{\alpha}+\frac m2,$$
%with equality if and only if $g$ is the Riemannian Schwarzschild metric.
%Observe that  $C_g(\Sigma)$ will always be less than $\frac{3m}{2}$, and could be smaller
% than $m$, depending on $\alpha$. 
\vspace{.4cm}

{\it Appendix.} In this appendix we include the detailed argument for Lemma 8 for weak solutions of inverse mean curvature flow in $\mathbb{R}^n$. \vspace{.2cm}

{\bf Lemma A1.} Let $\Omega\subset \mathbb{R}^n$ be open and smoothly bounded, $u:\Omega^c\rightarrow \mathbb{R}_+$ a proper function of class $C^2$ with $u_{|\partial \Omega}\equiv 0$. 

Let $t>0$, $\Omega_t=\{u\leq t\}$ and $\Phi:(0,t)\rightarrow \mathbb{R}_+$ be Lipschitz, with compact support in $(0,t)$. Then with $\varphi=\Phi\circ u:\Omega_t\rightarrow \mathbb{R}$ we have:
$$-\int_{\Omega_t}\nabla \varphi\cdot \nu Hdx
=\int_{\Omega_t}\varphi (H^2-|A|^2)dx$$
where $\nu, H,A$ denote the unit outward normal, mean curvature and second fundamental form of the level sets of $u$.
\vspace{.2cm}

{\it Proof.} (In this proof $\nabla f$, $Hf$ and $\Delta f$ denote the standard gradient, Hessian and Laplace operators on $\mathbb{R}^n$.)\vspace{.2cm}

(i) Suppose first $u\in C^3$ and $\Phi\in C^1$.  Recall that by Sard's Theorem we have for a.e. $t>0$: the level set $\Sigma_t=\{u=t\}$ is regular, i.e. $\nabla u\neq 0$ on the $C^2$ submanifold $\Sigma_t$. Let $U\subset \Omega^c$ be the open subset where $\nabla u\neq 0$. In $U$, define $f=-(|\nabla u|)^{-1}$. Consider a $ C^2$ level set $\Sigma$ of $u$ with outward unit normal $\nu$. Then in $\Sigma \cap U$ we have the pointwise identity:
$$\nu\cdot \nabla H=|\nabla u|\Delta_{\Sigma}f-|A|^2$$
Here $\Delta_{\Sigma}$ is the Laplace-Beltrami operator of $\Sigma$ in the induced metric:
$$\Delta_{\Sigma}f:=\Delta f-Hf(\nu,\nu)-H\partial_{\nu}f.$$
In particular, since $\Sigma$ is a level set of $u$:
$$\Delta_{\Sigma}u=0,\quad \Delta u-Hu(\nu,\nu)=H\partial_{\nu}u=H|\nabla u|.$$
This is used in the calculation below to `reduce the order' of certain terms.\vspace{.2cm}

(i-a)  In general $Hu(e,e)=\langle \nabla^{\Sigma}_e\nabla^{\Sigma}u_{|\Sigma},e\rangle+(\partial_{\nu}u)\langle \nabla_{e}\nu,e\rangle$ for $e\in T\Sigma$, or:
$$Hu_{|T\Sigma}=H^{\Sigma}(u_{|\Sigma})+(\partial_{\nu}u)A,$$
with $H^{\Sigma}$ the Hessian for the induced Riemannian metric on $\Sigma$. In case $\Sigma$ is a level set of $u$: $Hu_{|T\Sigma}=(\partial_{\nu}u)A$. Thus (with $(e_i)$ any local o.n. frame for $T\Sigma$):
$$|Hu|^2=\sum_{i,j}Hu(e_i,e_j)^2+2\sum_iHu(e_i,\nu)^2+Hu(\nu,\nu)^2=|\nabla u|^2|A|^2+2|\nabla|\nabla u||^2-Hu(\nu,\nu)^2,$$
(using $\nabla |\nabla u|=Hu(\nu)$ for the self-adjoint Hessian operator $Hu$.) We conclude:
$$|\nabla u|^2|A|^2=|Hu|^2-2|\nabla |\nabla u||^2+Hu(\nu,\nu)^2=|Hu|^2-2H^2u(\nu,\nu)+Hu(\nu,\nu)^2$$
and observe that $|Hu|^2\geq H^2u(\nu,\nu)\geq H(\nu,\nu)^2$. Here:
$$ H^2u(\nu,\nu)=u_{ab}u_{bc}\nu^a\nu^c,Hu(\nu,\nu)^2=u_{ab}u_{cd}\nu^a\nu^b\nu^c\nu^d.$$\vspace{.2cm}

(i-b) By direct calculation we find:
$$|\nabla u|\nabla u\cdot \nabla H=\nabla u\cdot \nabla (\Delta u)-|\nabla u|D^3u(\nu,\nu,\nu)$$
$$-(\Delta u-Hu(\nu,\nu))-2H^2u(\nu,\nu)+2Hu(\nu,\nu)^2.$$
and
$$|\nabla u|^3\Delta_{\Sigma}f=\nabla u\cdot \nabla (\Delta u)-|\nabla u|D^3u(\nu,\nu,\nu)+|\nabla u|^2|A|^2$$
$$-H|\nabla u|Hu(\nu,\nu)-2(H^2u(\nu,\nu)-Hu(\nu,\nu)^2).$$
Thus, using $\Delta_{\Sigma}u=0$ we find:
$$|\nabla u|\nabla u\cdot \nabla H-|\nabla u|^3\Delta_{\Sigma}f+|\nabla u|^2|A|^2=0.$$
Dividing both sides by $|\nabla u|^2$ (in $\Sigma\cap U$) concludes the proof of (i).\vspace{.2cm}

(ii) For $\Sigma$ a regular level set of $u$ ($\nabla u\neq 0$ on $\Sigma$) we have from (i):
$$\int_{\Sigma}\frac 1{|\nabla u|}\nu\cdot \nabla Hd\sigma
=\int_{\Sigma} \Delta_{\Sigma}fd\sigma-\int_{\Sigma}\frac 1{|\nabla u|}|A|^2d\sigma,$$
where the first term on the right hand side vanishes. Since $\nabla u\neq 0$ on $\Sigma_s$ for a.e. $s>0$, the coarea formula  and (i) give:
$$\int_{\Omega_t}\varphi \nu\cdot \nabla Hdx=\int_0^t\Phi(s)(\int_{\Sigma_s}
\frac 1{|\nabla u|}\nu\cdot \nabla Hd\sigma_s)ds
=-\int_0^t\Phi(s)(\int_{\Sigma_s}\frac 1{|\nabla u|}|A|^2d\sigma_s)ds
=-\int_{\Omega_t}\varphi |A|^2dx.$$

\vspace{.2cm}

(iii) In the open set $U=\{\nabla u\neq 0\}$ we have, pointwise:
$$\mbox{div }(\varphi H \nu)=(\Phi'\circ u)|\nabla u|H +\varphi \nu\cdot \nabla H+\varphi H\mbox{div }\nu.$$
Now, if $H$ denotes the weak mean curvature (see \cite{huiskenilmanen} for definitions) we see that:
$$\int_{\Omega_t} \psi\mbox{div }\nu dx=\int_{\Omega_t} H\psi dx, \quad \forall \psi\in W^{1,2}_0(\Omega_t).$$
Indeed for the weak mean curvature we have on any regular level set $\Sigma$:
$$\int_{\Sigma}g \mbox{div}_{\Sigma}Xd\sigma=\int_{\Sigma}g H(X\cdot \nu) d\sigma,$$
for all smooth functions $g$ and smooth $\mathbb{R}^n$-valued vector fields $X$ on $\Sigma$, and by approximation also for $g\in W^{1,2}(\Sigma)$, $X\in L^2(\Sigma;\mathbb{}R^n)$ (since $H\in L^{\infty}(\Omega_t)$). In particular,  we may take $X=\nu$.

Here $\mbox{div}_{\Sigma}X=\sum_i\nabla_{e_i}X\cdot e_i$, while $\mbox{div }X=\partial_aX^a$, for $X=X^a\partial_a$ (a vector field in $\mathbb{R}^n$, in the standard basis), so for the unit normal
$\mbox{div }\nu=\mbox{div}_{\Sigma}\nu$ (pointwise on $U$, or in $W^{-1,2}$ in $\Omega_t$.) Now for $\psi\in W^{1,2}_0(\Omega_t)$ let $g=|\nabla u|\psi$ (note $|\nabla^{\Sigma}|\nabla u||\leq |Hu|$ on regular $\Sigma_s$, hence for a.e. $s$) and use the coarea formula:
$$\int_{\Omega_t} \psi\mbox{div }\nu dx=
\int_0^t(\int_{\Sigma_s}|\nabla u|\psi\mbox{div }\nu d\sigma_s)ds=
\int_0^t(\int_{\Sigma_s}|\nabla u|\psi Hd\sigma_s)ds=\int_{\Omega_t}\psi Hdx,$$
as asserted.\vspace{.2cm}

Since $\varphi H \in L^{\infty}\cap W^{1,2}_0$ (assuming $u\in C^3$), we conclude:
$$\mbox{div }\nu=H\mbox{ in }W^{-1,2}(\Omega_t), \mbox{ hence }H\varphi\mbox{div }\nu=\varphi H^2\mbox{ in }W^{-1,2}(\Omega_t).$$
This implies we have in the sense of distributions in $W^{-1,2}$:
$$\mbox{div }(\varphi H \nu)=(\Phi'\circ u)|\nabla u|H +\varphi \nu\cdot \nabla H+\varphi H^2.$$
Integration over $\Omega_t$ yields:
$$0=\int_{\Omega_t}(\Phi'\circ u)|\nabla u|Hdx+\int_{\Omega_t} \varphi\nu\cdot \nabla Hdx+\int_{\Omega_t}\varphi H^2dx.$$
Thus, using (ii) for the second term on the right:
$$-\int_{\Omega_t}\nabla \varphi\cdot \nu Hdx=-\int_{\Omega_t}\varphi |A|^2dx+\int_{\Omega_t}\varphi H^2dx,$$
as claimed in the statement of the lemma. By approximation, this is also true if $u\in C^2(\Omega^c)$ and
$\Phi$ is Lipschitz in $(0,t)$, with compact support.\vspace{.3cm}

{\it Remark.} The authors have been informed that the lemma has been known to G. Huisken for some time, but no proof has been published (cf. Theorem 6 in \cite{guanli}).

\end{document}